\documentclass[10pt,preprint,reqno]{amsart}

\usepackage{amsmath,amssymb,amsfonts,amsthm}

\usepackage[utf8]{inputenc}

\usepackage{amscd, calc, tikz, tikz-cd}
\usepackage{enumerate,graphicx,psfrag}
\usepackage{mathrsfs}
\usepackage{hyperref}

\usetikzlibrary{arrows}

\usepackage{color} 
\usepackage{pstricks,pdftricks}

\newcommand \vanish[1]{}

\newtheorem{thm}{Theorem}[section]
\newtheorem{lem}[thm]{Lemma}
\newtheorem{prop}[thm]{Proposition}

\newtheorem{cor}[thm]{Corollary}

\theoremstyle{definition}
\newtheorem{defn}[thm]{Definition}
\newtheorem{nota}[thm]{Notation}

\newtheorem{example}[thm]{Example}

\allowdisplaybreaks

\newcommand\ds\displaystyle
\newcommand\ts\textstyle

\renewcommand{\phi}{\varphi}                 
\renewcommand{\epsilon}{\varepsilon}
\newcommand\eset{\varnothing}

\renewcommand\emptyset\eset
\renewcommand\ell{l}

\newcommand\cupdot {\mbox{\hspace{.15em}$\cup$\hspace{-.47em}$\cdot$\hspace{.4em}}}

\newcommand\supp{\operatorname{supp}}

\usepackage[top=1.25in, bottom=1.25in, left=1in, right=1in]{geometry}                		

\usepackage{tikz-feynman}
\usetikzlibrary{decorations.pathmorphing}

\begin{document}
\title{Classification of Doubly Distributive skew Hyperfields and Stringent hypergroups}

\author{Nathan Bowler \and Ting Su}

\email{Nathan.Bowler@uni-hamburg.de}

\address{Department of Mathematics, Universit\"at Hamburg, Germany}

\email{ting.su@uni-hamburg.de}

\address{Department of Mathematics, Universit\"at Hamburg, Germany}


\begin{abstract} A hypergroup is {\em stringent} if $a\boxplus b$ is a singleton whenever $a \neq -b$. A hyperfield is stringent if the underlying additive hypergroup is. Every {\em doubly distributive} skew hyperfield is stringent, but not vice versa. We present a classification of stringent hypergroups, from which a classification of doubly distributive skew hyperfields follows. It follows from our classification that every such hyperfield is a quotient of a skew field.

\end{abstract}

\maketitle

\section{Introduction}
The notion of {\em hyperfield} was first introduced by Krasner in \cite{Kra57, Kra83}. It is an algebraic structure similar to a field except that its addition $\boxplus$ is multivalued. In \cite{Viro10}, Viro provided an excellent introduction to and motivation for hyperfields and introduced several good examples of hyperfields, including the {\em tropical hyperfield} $\mathbb{T}_+$, the {\em tropical real hyperfield} $\mathbb{TR}$ and the {\em ultratriangle hyperfield} $\mathbb{T}\triangle$. Viro has also illustrated the utility of $\mathbb{T}_+$ for the foundations of tropical geometry in several interesting papers (cf. \cite{Viro10, Viro11}).




In \cite{Bak16}, Baker and Bowler presented an algebraic framework which simultaneously generalizes the notion of linear subspaces, matroids, oriented matroids, and valuated matroids, and called the resulting objects {\em matroids over hyperfields}. A matroid over a field $F$ corresponds to a subspace of some $F^n$. A $\mathbb{K}$-matroid is just a matroid. An $\mathbb{S}$-matroid is an oriented matroid. And a $\mathbb{T}\triangle$-matroid is a valuated matroid, as defined in \cite{DW92b}.

Baker and Bowler also provided two natural notions of matroids over a hyperfield $F$, {\em weak $F$-matroids} and {\em strong $F$-matroids}, and showed that the two notions coincide when $F$ has a property called {\em double distributivity}. A hyperfield $F$ is {\bf doubly distributive} if $(a \boxplus b)(c \boxplus d)=ac \boxplus ad \boxplus bc \boxplus bd$ for any $a,b,c,d \in F$. Fields, $\mathbb{K}$,  $\mathbb{S}$ and $\mathbb{T}\triangle$ are all doubly distributive. So too are the other two hyperfields mentioned above, $\mathbb{T}_+$ and $\mathbb{TR}$.

It is these the results in tropical geometry and matroid theory which motivate our interest in doubly distributive hyperfields. More generally, we are also interested in doubly distributive hyperrings, which were also analysed by Baker and Bowler. In fact, rather than just hyperfields, they worked with a more general kind of algebraic object known as tracts (cf. \cite{Bak17}). The other important example of tracts other than hyperfields is given by partial fields, which have also been the subject of much fruitful study. Baker and Bowler defined a special class of tracts called {\em partial hyperfields}, objects based on hyperrings which generalize both hyperfields and partial fields in a natural way. The property of double distributivity also extends to hyperrings and thus to partial hyperfields.

We will classify the doubly distributive skew hyperfields in Section~\ref{sect.class.DDsH}. The classification itself will be described in Section~\ref{sect.class.ssH}, but has the following important consequence:

\begin{defn}
A {\bf valuation} $\nu$ of a skew hyperfield $F$ is a map from $F$ to $G\cup \{-\infty\}$, where $(G, <)$ is a linearly ordered group, satisfying
\begin{enumerate}
\item $\nu(x) = -\infty$ if and only if $x = 0$.
\item $\nu(xy) = \nu(x)\cdot \nu(y)$.
\item $\nu(x) > \nu(y)$ implies $x\boxplus y = \{x\}$.
\end{enumerate}
\end{defn}

\begin{thm}
For every doubly distributive skew hyperfield $F$, there is always a valuation $\nu$ of $F$ such that $\nu^{-1}(1_G)$ is either the Krasner hyperfield, or the sign hyperfield, or a skew field.
\end{thm}
This compact description is from the paper \cite{BP19}.

In particular, since any nontrivial ordered group is infinite, it follows from our results that the only finite doubly distributive hyperfields are the Krasner hyperfield, the sign hyperfield and the finite fields.

This classification has a number of applications. For example, we use it in Section~\ref{sect.quotient} to show that any doubly distributive skew hyperfield is a quotient of a skew field. Bowler and Pendavingh used it in \cite{BP19} to show that any doubly distributive skew hyperfield is {\em perfect} and to provide vector axioms for matroids over such skew hyperfields.

Our classification uses a property of the underlying hypergroup which we call {\em stringency}. A hyperfield $F$ is {\bf stringent}, if $a \boxplus b$ is a singleton whenever $a \neq -b$.

\begin{prop}\label{ddsingleton} Every doubly distributive skew hyperfield is stringent.
\end{prop}

\begin{proof} Let $F$ be a doubly distributive skew hyperfield. Let $a, b \in F^\times$ be such that $a\neq -b$. Let $x, y \in F^\times$ be such that $x, y\in a\boxplus b$. By double distributivity, we have
$$(a\boxplus b)(x^{-1} \boxplus -y^{-1}) = (a\boxplus b)\cdot x^{-1} \boxplus (a\boxplus b)\cdot (-y^{-1}) \supseteq x \cdot x^{-1} \boxplus  y\cdot (- y^{-1}) = 1\boxplus -1 \ni 0.$$

As $a\neq -b$, then $x^{-1}=y^{-1}$, and so $x = y$. So $a \boxplus b$ is a singleton if $a\neq -b$.
\end{proof}

However, not every stringent skew hyperfield is doubly distributive. The following is a counterexample.

\begin{example}\label{stringentnotdd} Let $F : = \mathbb{Z} \cup \{-\infty\}$ be the stringent hyperfield with multiplication given by $a\odot b = a+b$ and multiplicative identity $0$. Hyperaddition is given by
$$a\boxplus b = \begin{cases}
\{\max(a, b)\} & \text{ if $a \neq b$,} \\
\{c \,|\, c< a \} & \text{ if $a = b$,}
\end{cases}
$$
so that the additive identity is $-\infty$. Here we use the standard total order on $\mathbb{Z}$ and set $-\infty < x$ for all $x\in \mathbb{Z}$.

$F$ is not doubly distributive because
\begin{align*}
(0 \boxplus 0) \odot (0 \boxplus 0) & = \{z \,|\, z < 0\} \odot \{z \,|\, z < 0\} = \{z \,|\, z < -1\}, \\
0 \boxplus 0 \boxplus 0 \boxplus 0 & = \{z \,|\, z < 0\} \boxplus \{z \,|\, z < 0\} = \{z \,|\, z < 0\}. 
\end{align*}
\end{example}

We use our classification of stringent skew hyperfields to derive a classification of stringent skew hyperrings in Section~\ref{sect.class.SSHR}. However, this does not give a classification of doubly distributive skew hyperrings, since not every doubly distributive skew hyperring is stringent (see Example~\ref{coutex.ddhyperringnotstringent}).

In fact, we classify all stringent hypergroups, and our classification of doubly distributive skew hyperfields follows from this.

\begin{defn}\label{hypergroup.constr}
Let $(G, <)$ be a totally ordered set, let $(F_g \,|\, g\in G)$ be a family of hypergroups with a common identity element $0$ in each $F_g$ but otherwise disjoint, and let $\psi$ be the surjective function from $\bigcup_{g\in G} F_g^\times$ to $G$ sending $f$ in $F_g^\times$ to $g$. We denote the hyperaddition of $F_g$ by $\boxplus_g$. For any $g \in G$ we denote by $g \downarrow$ the set of $h \in G$ with $h < g$.

Then the \textbf{wedge sum} $F = \bigvee_{g\in G}{F_g}$ is the hypergroup with ground set $\bigcup_{g\in G} F_g$ and hyperaddition given by
\begin{align*}
x\boxplus 0 & = 0 \boxplus x = \{x\}, \\
x \boxplus y & = \begin{cases}
\{x\} & \text{if $\psi(x) > \psi(y)$,} \\ 
\{y\} & \text{if $\psi(x) < \psi(y)$,} \\ 
x \boxplus_{\psi(x)} y & \text{if $\psi(x) = \psi(y)$ and $0 \not \in x \boxplus_{\psi(x)} y$,}\\ 
(x \boxplus_{\psi(x)} y) \cup \psi^{-1}( \psi(x)\downarrow) & \text{if $ \psi(x) = \psi(y)$ and $0 \in x \boxplus_{\psi(x)} y$. }
\end{cases}
\end{align*}

We can also define $\bigcup_{g\in G} F_g$ up to isomorphism if the $F_g's$ don't have the same identity or aren't otherwise disjoint, by replacing the $F_g's$ with suitably chosen isomorphic copies.
\end{defn}

We will show in Section~\ref{sect.class.sH} that this construction always yields a hypergroup, and we classify the stringent hypergroups as follows:

\begin{thm}\label{classgroups}
Every stringent hypergroup is a wedge sum $\bigvee_{g\in G}{F_g}$ where each $F_g$ is either a copy of the Krasner hypergroup, or a copy of the sign hypergroup, or a group. 
\end{thm}

This classification of hypergroups is used to derive the classification of doubly distributive skew hyperfields discussed above.

\subsection{Structure of the paper} 
After the classification of stringent hypergroups in Section~\ref{sect.class.sH}, we show in Section~\ref{sect.class.ssH} that every stringent skew hyperfield arises from a short exact sequence of groups, where the first group in the sequence is the multiplicative group of either the Krasner hyperfield or the sign hyperfield or a skew field, and the last group in the sequence is a totally ordered group. The underlying additive hypergroup is a wedge sum of isomorphic copies of hypergroups. Then we present the classification of doubly distributive skew hyperfields in Section~\ref{sect.class.DDsH} following from the classification of stringent skew hyperfields. We show the surprising result that every stringent skew hyperring is either a skew ring or a stringent skew hyperfield in Section~\ref{sect.class.SSHR}. We use our classification to show that every stringent skew hyperfield is a quotient of a skew field by some normal subgroup in Section~\ref{sect.quotient}. In Appendix~\ref{appendix} we present a proof that a construction really gives a skew field and in Appendix~\ref{app2} we talk about the semirings associated to doubly distributive hyperfields.

\subsection*{Acknowledgements}
We thank Matthew Baker and Laura Anderson (second author's PhD advisor) for introducing the two authors to each other. We thank Laura Anderson and Tom Zaslavsky, who gave us important comments on early versions of the work. Thanks also to Pascal Gollin for asking whether our classification might hold for all stringent hypergroups.
\section{Background}

\begin{nota} Throughout $G$ and $H$ denotes groups. 

For a hypergroup (or skew hyperring) $S$, $S^\times$ denotes $S-\{0\}$.

For a function $f$ from a hypergroup (or skew hyperring) $A$ to a hypergroup (or skew hyperring) $B$, $\supp(f)$ denotes the set of support of $f$ (the elements of $A$ where the function value is not zero).

\end{nota}

\subsection{Hypergroups, hyperrings and hyperfields}

\begin{defn}\label{hyperoperation} A \textbf{hyperoperation} on a set $S$ is a map $\boxplus$ from $S \times S$ to the collection of non-empty subsets of $S$.

If $A$, $B$ are non-empty subsets of $S$, we define
$$A \boxplus B := \bigcup_{a\in A, b\in B} a\boxplus b $$
and we say that $\boxplus$ is \textbf{associative} if $a\boxplus(b\boxplus c)=(a\boxplus b) \boxplus c$ for all $a, b, c \in S$.
\end{defn}

All hyperoperations in this paper will be associative.

\begin{defn}\label{hypergroup} \cite{Viro10}
A \textbf{hypergroup} is a tuple $(G,\boxplus, 0)$ where $\boxplus$ is an associative hyperoperation on $G$ such that:
\begin{enumerate}[(1)]
\item $0\boxplus x = x \boxplus 0 = \{x\}$ for all $x \in G$. 

\item For every $x \in G$ there is a unique element $x'$ of $G$  such that $0 \in x \boxplus x'$ and there is a unique element $x''$ of $G$  such that $0 \in x'' \boxplus x$. Furthermore, $x'= x''$. This element is denoted by $-x$ and called the {\bf hyperinverse} of $x$.

\item\label{invertibility}(Invertibility of sums) $x \in y \boxplus z$ if and only if $-x \in -z \boxplus -y$.

A hypergroup is said to be \textbf{commutative} if
\item $x \in y \boxplus z$ if and only if $x \in z \boxplus y$.
\end{enumerate}
\end{defn}

\begin{thm}\cite{Viro10} In Definition~\ref{hypergroup}, the axiom (\ref{invertibility}) can be replaced by

(Reversibility property) $x\in y \boxplus z$ implies $y \in x\boxplus -z$ and $z \in -y \boxplus x$.
\end{thm}

The Reversibility property was introduced by Marshall in \cite{Mar06}.

\begin{defn}
A \textbf{skew hyperring} is a tuple $(R,\odot, \boxplus, 1, 0)$ such that:
\begin{enumerate}[(1)]
\item $(R, \odot, 1)$ is a monoid.

\item $(R, \boxplus, 0)$ is a commutative hypergroup.

\item (Absorption rule) $x\odot 0 = 0\odot x = 0$ for all $x \in R$.

\item (Distributive Law) $a \odot (x \boxplus y )= (a \odot x) \boxplus (a \odot y)$ and $(x \boxplus y ) \odot a = (x \odot a) \boxplus (y \odot a)$ for all $a, x, y \in R$.
\end{enumerate}

A \textbf{hyperring} is a skew hyperring with commutative multiplication.

A skew hyperring $F$ is called a \textbf{skew hyperfield} if $0 \neq 1$ and every non-zero element of $F$ has a multiplicative inverse.

A \textbf{hyperfield} is then a skew hyperfield with commutative multiplication.
\end{defn}

\begin{defn} Let $F$ and $G$ be skew hyperrings. We may define a skew hyperring $F \times G$ with $(x_1, y_1) \boxplus (x_2, y_2)$ defined as $(x_1 \boxplus_F x_2) \times (y_1 \boxplus_G y_2)$ and multiplication defined pointwise. Its additive identity is $(0_F, 0_G)$ and its multiplicative identity is $(1_F, 1_G)$. We call $F \times G$ the \textbf{product} of $F$ and $G$.
\end{defn}

Let $x, y \in F$, we will sometimes write $xy$ instead of $x\odot y$ if there is no risk of confusion.

\begin{example} In \cite{Viro10}, Viro provided a good introduction to hyperfields. Several of the following hyperfields were first introduced there.

\begin{enumerate}[(1)]
\item If $F$ is a field, then $F$ is a hyperfield with $a\odot b = a \cdot b$ and $a\boxplus b = \{a + b\}$, for any $a, b \in F$.

\item The \textbf{Krasner hyperfield} $\mathbb{K}:=\{0, 1\}$ has the usual multiplication rule and hyperaddition is defined by $0\boxplus x = \{x\}$ for $x \in \mathbb{K}$ and $1\boxplus 1 = \{0,1\}$.

\item The \textbf{sign hyperfield} $\mathbb{S} := \{0, 1, -1\}$ has the usual multiplication rule and hyperaddition is defined by $0\boxplus x = \{x\}, x\boxplus x = \{x\}$ for $x \in \mathbb{S}$, and $1\boxplus -1 = \{0,1, -1\}$.

\item The \textbf{triangle hyperfield} $\triangle := \mathbb{R}_{\geq 0}$ has the usual multiplication rule and hyperaddition is defined by $x\boxplus y = \{z \, | \, |x-y|\leq z \leq x+y \}$.

\item The {\bf tropical hyperfield} $\mathbb{T}_+ : = \mathbb{R}\cup \{-\infty\}$ has multiplication defined by $x\odot y = x + y$ (with $-\infty$ as an absorbing element), for $x, y \in \mathbb{T}_+$. Hyperaddition is defined by
$$x \boxplus y =\begin{cases}
\{\max(x, y)\}	&\text{ if $x\neq y$,}\\
\{z \, | \, z \leq x\} &\text{ if $x = y$.}
\end{cases}
$$
Here we use the standard total order on $\mathbb{R}$ and set $-\infty < x$ for all $x \in \mathbb{R}$. The additive identity is $-\infty$ and the multiplicative identity is $0$. 

\item The {\bf tropical phase hyperfield} $\Phi: = S^1\cup \{0\}$ has the usual multiplication rule and hyperaddition is defined by $0\boxplus x = \{x\}$, $x\boxplus -x = S^1\cup \{0\}$ and $x\boxplus y = \{\frac{ax+by}{|ax+by|}\, |\, a,b\in \mathbb{R}_{\geq 0}, a+b \neq 0 \}$ for $x, y \in S^1$ with $y \neq -x$.\footnote{This is called phase hyperfield in Viro's paper, but more recent papers have often worked with the phase hyperfield (\ref{phasehyperfield}) described next. The confusion on this point is exacerbated by the fact that Viro incorrectly claims that his phase hyperfield is the same as the quotient hyperfield of the complex numbers by the positive real numbers, but this construction actually gives the hyperfield  (\ref{phasehyperfield}).}

\item\label{phasehyperfield} The \textbf{phase hyperfield} $\mathbb{P} : = S^1\cup \{0\}$ has the usual multiplication rule and hyperaddition is defined by $0\boxplus x = \{x\}$, $x\boxplus -x = \{x, -x, 0\}$ and $x\boxplus y =\{\frac{ax+by}{|ax+by|}\, |\, a,b\in \mathbb{R}_{> 0}\}$ for $x, y \in S^1$ with $y \neq -x$.

\item The \textbf{tropical real hyperfield} $\mathbb{TR} := \mathbb{R}$ has the usual multiplication rule and hyperaddition is defined by
$$x \boxplus y =\begin{cases}
\{x\}	&\text{ if $ |x| > |y|$,}\\
\{y\}	&\text{ if $ |x| < |y|$,}\\
\{x\}   &\text{ if $x = y$,}\\
\{z \, | \, |z| \leq |x|\} &\text{ if $x = - y$.}
\end{cases}
$$

\item The \textbf{tropical complex hyperfield} $\mathbb{TC} := \mathbb{C}$ has the usual multiplication rule and hyperaddition is defined by
$$x \boxplus y =\begin{cases}
\{x\}	&\text{if $ |x| > |y|$,}\\
\{y\}	&\text{if $ |x| < |y|$,}\\
\{|x|\dfrac{a x + b y}{|a x + b y|} \,|\, a, b \in \mathbb{R}_{\geq 0}, a+b\neq 0\}   &\text{if $|x| = |y|$ and $x\neq -y$,}\\
\{z \,|\, |z| \leq |x|\} &\text{if $x = - y$.}
\end{cases}
$$

\item The \textbf{ultratriangle hyperfield} $\mathbb{T}\triangle := \mathbb{R}_{\geq 0}$ (denoted by $\mathbb{Y}_{\times}$ in \cite{Viro10} and $\mathbb{T}$ in \cite{Bak17}) has the usual multiplication rule and hyperaddition is defined by
$$x \boxplus y =\begin{cases}
\{\max(x, y)\}	&\text{ if $x \neq y$,}\\
\{z \, | \, z \leq x \} &\text{ if $x = y$.}
\end{cases}
$$
\end{enumerate}
\end{example}

\begin{defn}\label{dd} \cite{Viro10, Bak17}
A skew hyperring $R$ is said to be \textbf{doubly distributive} if for any $a$, $b$, $c$ and $d$ in $R$, we have
$(a \boxplus b)(c \boxplus d)=ac \boxplus ad \boxplus bc \boxplus bd.$ 
\end{defn}

\begin{example} Fields, $\mathbb{K}$, $\mathbb{S}$, $\mathbb{T}_{+}$, $\mathbb{TR}$, $\mathbb{T}\triangle$ are all doubly distributive, but $\triangle$, $\mathbb{P}$, $\Phi$ and $\mathbb{TC}$ are not doubly distributive. 
\end{example}

\begin{defn}\label{stringent} A hypergroup $G$ is said to be \textbf{stringent} if for any $a, b \in G$ the set $a\boxplus b$ is a singleton whenever $a\neq -b$.

A skew hyperring is said to be \textbf{stringent} if the underlying hypergroup $F$ is stringent.
\end{defn}

\subsection{Homomorphism}

\begin{defn}\label{hyp.homo} \cite{Bak16, Pen18} 
A \textbf{hypergroup homomorphism} is a map $f : G \rightarrow H$ such that $f(0)= 0$ and $f(x \boxplus y) \subseteq f(x)\boxplus f(y)$
for all $x, y \in G$.

A \textbf{skew hyperring homomorphism} is a map $f : R \rightarrow S$ which is a homomorphism of additive commutative hypergroups as well as a homomorphism of multiplicative monoids (i.e., $f(1) = 1$ and $f(x\odot y) = f(x)\odot f(y)$ for $x, y \in R$).

A \textbf{skew hyperfield homomorphism} is a homomorphism of the underlying skew hyperrings.

A \textbf{hypergroup (resp. skew hyperring, skew hyperfield) isomorphism} is a bijection $f : G \rightarrow H$ which is a hypergroup (resp. skew hyperring, skew hyperfield) homomorphism and whose inverse is also a hypergroup (resp. skew hyperring, skew hyperfield) homomorphism.
\end{defn}

\begin{example} The map $\exp : \mathbb{T}_+ \rightarrow \mathbb{T}\triangle$ is a hyperfield isomorphism.
\end{example}

\section{Classification of stringent hypergroups}\label{sect.class.sH}

Our aim in this section is to prove Theorem \ref{classgroups}, the Classification Theorem for stringent hypergroups. We will work with the definition of wedge sums given as Definition \ref{hypergroup.constr}. First we will show that $F := \bigvee_{g \in G} F_g$ is indeed a hypergroup.

\begin{lem} \label{comm.HG}
$F$ is again a hypergroup. If every hypergroup in $(F_g \,|\, g\in G)$ is stringent, then so is $F$. If every hypergroup in $(F_g \,|\, g\in G)$ is commutative, then so is $F$. 
\end{lem}
\begin{proof}
For associativity, suppose we have $x_1, x_2, x_3 \in F$. If any of them is 0, then associativity is clear, so suppose that each $x_i$ is in $H$. If one of the elements $\psi(x_i)$ of $G$, say $\psi(x_{i_0})$, is bigger than the others, then $x_1 \boxplus (x_2 \boxplus x_3) = \{x_{i_0}\} = (x_1 \boxplus x_2) \boxplus x_3$. If one of the $\psi(x_i)$ is smaller than the others, then both $x_1 \boxplus (x_2 \boxplus x_3)$ and $(x_1 \boxplus x_2) \boxplus x_3$ evaluate to the sum of the other two $x_j$. So we may suppose that all $\psi(x_i)$ are equal, taking the common value $g$. If $0 \not \in x_1 \boxplus_g x_2 \boxplus_g x_3$, then both $x_1 \boxplus (x_2 \boxplus x_3)$ and $(x_1 \boxplus x_2) \boxplus x_3$ evaluate to $x_1 \boxplus_g x_2 \boxplus_g x_3$, whereas if $0 \in x_1 \boxplus_g x_2 \boxplus_g x_3$, then both evaluate to $(x_1 \boxplus_g x_2 \boxplus_g x_3) \cup \psi^{-1}(g \downarrow)$. The hyperinverse of 0 is 0 and the hyperinverse of any other $x$ is its hyperinverse in $F_{\psi(x)}$, and $0$ is the additive identity. 

For invertibility of sums, suppose we have $x, y, z \in F$. We would like to show that $x\in y\boxplus z$ if and only if $-x \in -z \boxplus -y$. It suffices to prove one direction, say if $x\in y\boxplus z$, then $-x \in -z \boxplus -y$. If $\psi(y)< \psi(z)$, then $x\in y\boxplus z = \{z\}$ and $\psi(-y) < \psi(-z)$. So $- z\boxplus -y = \{-z\} = \{-x\}$. Similarly, we have if $\psi(y)> \psi(z)$, then $- z\boxplus -y = \{-x\}$. If $\psi(x) = \psi(y) = \psi(z)$, then the statement holds by the reversibility of the hypergroup $(F_{\psi(y)}, \boxplus_{\psi(y)}, 0)$. Otherwise we have $\psi(x) < \psi(y) = \psi(z)$, and so $y = -z$. Then $\psi(-x) < \psi(-y) = \psi(-z)$ and so $-x \in -z\boxplus -y$. 

Then, we would like to show that $F$ is stringent if every hypergroup $F_g$ in $(F_g \,|\, g\in G)$ is. By definition of $F$, we just need to show that for any $x, y \in F$ with $\psi(x) = \psi(y)$ and $0 \not \in x \boxplus_{\psi(x)} y$, $x\boxplus y$ is a singleton. As $F_{\psi(x)}$ is stringent and $0 \not \in x \boxplus_{\psi(x)} y$, then $x \boxplus_{\psi(x)} y$ is a singleton. So $x\boxplus y = x \boxplus_{\psi(x)} y$ is also a singleton.

Finally, it is clear that $\boxplus$ is commutative if each $\boxplus_g$ is commutative.
\end{proof}

Now we begin the proof of the Classification Theorem. We first introduce a useful lemma. Note that this lemma automatically holds for stringent commutative hypergroups, so readers only interested in that case may skip the proof.

\begin{lem}\label{x+y = y+x}
Let $F$ be a stringent hypergroup. If $y\in x\boxplus y$, then $y\in y\boxplus x$.
\end{lem}

\begin{proof} We will divide the proof into four cases.

\emph{Case 1: } If $x = y$, this is immediate.

\emph{Case 2: } If $x = -y$, then by reversibility we get
$$y\in x\boxplus y \Rightarrow y\in -y \boxplus y \Rightarrow y \in y \boxplus y \Rightarrow y\in y\boxplus -y \Rightarrow y\in y\boxplus x.$$

\emph{Case 3: } If $y = -y$, then by reversibility and case 2 we get
$$y\in x\boxplus y \Rightarrow x \in y \boxplus -y \Rightarrow x \in -y \boxplus y \Rightarrow y\in y\boxplus x.$$

\emph{Case 4: } Now we suppose $x\notin \{y, -y\}$ and $y\neq -y$. Let $z\in F^\times$ be such that $y\boxplus x = \{z\}$ and let $t\in F^\times$ be such that $-y \boxplus -y = \{t\}$. Then by associativity we get
$$z \boxplus y \boxplus t = (y \boxplus x) \boxplus y \boxplus (-y \boxplus -y) =  y \boxplus (x \boxplus y) \boxplus -y \boxplus -y = y \boxplus y \boxplus -y \boxplus -y \ni 0.$$
So we get $0\in z \boxplus y \boxplus t$, thus $-z \in y\boxplus t$. As $t \in -y \boxplus -y$, we have $-y \in y\boxplus t$. So $-z, -y \in y\boxplus t$. Then by stringency we get either $z = y$ or $t = -y$. If $z = y$, then we are done. Now assume $t = -y$. Thus $-z\in y\boxplus t = y\boxplus -y$, and so $-y \in -y\boxplus z$. As $y\boxplus x = \{z\}$, we have $x\in -y\boxplus z$. So $-y, x\in -y\boxplus z$. Then by stringency we get either $x = -y$ or $z = y$. By case 2, the statement holds.
\end{proof}

Now we define a relation on $F^{\times}$ which roughly corresponds to the ordering of $G$.

\begin{defn}\label{def.less.reln.group}
We define a relation $<_F$ on $F^{\times}$ by $x <_F y$ if $x \boxplus y = y \boxplus x = \{y\}$ but $x \neq y$. 
\end{defn}

\begin{lem}
$<_F$ is a strict partial order on $F^{\times}$.
\end{lem}
\begin{proof}
Irreflexivity is built into the definition, so it remains to check transitivity. Suppose that $x <_F y <_F z$. Then $x \boxplus z = x \boxplus y \boxplus z = y \boxplus z = \{z\}$. Similarly, $z \boxplus x = \{z\}$. We cannot have $x = z$, since then $\{y\} = y \boxplus x = y \boxplus z = \{z\}$, so $y = z$, which is a contradiction.
\end{proof}

\begin{lem} \label{pm}
If $x <_F y$, then 
\begin{enumerate}
\item $\pm x <_F \pm y$.
\item for any $z \in F^{\times}$ we have either $x <_F z$ or $z <_F y$.
\end{enumerate}
\end{lem}
\begin{proof}
\begin{enumerate}
\item It suffices to prove that $-x <_F y$ by invertibility of sums. As $x<_F y$, then $x \neq -y$ since $0\in -y\boxplus y$. As $x\boxplus y = \{y\}$, then $y \in -x\boxplus y$. By stringency, $-x\boxplus y = \{y\}$. Similarly, $y\boxplus -x = \{y\}$. So $-x <_F y$.
\item By (1), we have $\pm x <_F \pm y$. Suppose that $z \not <_F y$. If $z\in \{ y, -y\}$, then we have $x <_F z$. Otherwise, $y \not \in z \boxplus y$ and $y\notin y\boxplus z$ by Lemma~\ref{x+y = y+x}. Then $0 \not \in z \boxplus y \boxplus -y$ and $0 \not \in -y \boxplus y \boxplus z$. So by stringency, we have $z \boxplus y \boxplus -y = \{z\}$ and $ -y \boxplus y \boxplus z = \{z\}$. However, $x \in y \boxplus -y$ and $x\in -y\boxplus y$, since $x <_F y$. So $z \boxplus x = \{z\}$ and $x\boxplus z = \{z\}$. Now if $z \neq x$ this implies that $x <_F z$, but if $z = x$ then we have $z <_F y$.
\end{enumerate}
\end{proof}

Now we define a relation $\sim_F$ on $F^{\times}$ by $x \sim_F y$ if and only if both $x \not <_F y$ and $y \not <_F x$.

\begin{lem}\label{eqv.reln.group}
$\sim_F$ is an equivalence relation.
\end{lem}
\begin{proof}
$\sim_F$ is clearly reflexive and symmetric. For transitivity, suppose that $x \sim_F y$ and $y \sim_F z$. If $x <_F z$ then either $x <_F y$, contradicting $x \sim_F y$, or else $y <_F z$, contradicting $y \sim_F z$, so this is impossible. Similarly we have $z \not <_F x$. So $x \sim_F z$.
\end{proof}

The following results are obvious and we will put them together.

\begin{lem}\label{eqv.set}
 \begin{enumerate}
 \item If $x \sim_F y <_F z$ or $x <_F y \sim_F z$, then $x <_F z$.
 \item The relation $<_F$ could lift to a relation (denoted by $<_F'$) on the set $G$ of $\sim_F$-equivalence classes and $(G, <_F')$ is a totally ordered set. 
\item For every $x\in F^\times$, $-x \sim_F x$.
\item Let $x, y, z \in F^{\times}$ with $x \neq -y$, $y \neq -z$ and $z \neq -x$. If $0 \in x \boxplus y \boxplus z$, then $x \sim_F y \sim_F z$.
\end{enumerate}
\end{lem}

\begin{proof} For (1), the proof is trivial. (1) implies (2).

(3) As $0\in x\boxplus -x$, we have $x \not<_F -x$ and $-x \not <_F x$. So $-x \sim_F x$.

(4) If not, then without loss of generality we have $x <_F y$, and so $-z \in x \boxplus y = \{y\}$, giving $y = -z$, contradicting our assumptions.
\end{proof}

\begin{lem}\label{addition}
Let $(x_i \, | \, i \in I)$ be a finite family of elements of $F$, and $z \in F$ with $x_i <_F z$ for all $i \in I$. Then for any $y \in \boxplus_{i \in I} x_i$ we have $y <_F z$. 
\end{lem}
\begin{proof}
It suffices to prove this when $I$ has just two elements, say $x_1$ and $x_2$, since the general result then follows by induction. Suppose $x_1, x_2 <_F z$ and $y \in x_1 \boxplus x_2$, then we have $$y \boxplus z \subseteq x_1 \boxplus x_2 \boxplus z = x_1 \boxplus z = \{z\}$$
and $$z \boxplus y \subseteq z \boxplus x_1 \boxplus x_2  = z \boxplus x_2 = \{z\}.$$ 
So $y \boxplus z = \{z\}$ and $z\boxplus y = \{z\}$. If $z \in x_1 \boxplus x_2$ then $-x_1 \in x_2 \boxplus -z = \{-z\}$, contradicting $x_1 <_F z$. So $z \not \in x_1 \boxplus x_2$, and so $z \neq y$. So $y<_F z$.
\end{proof}

It follows from the above results that the sum $x \boxplus y$ is given by $\{x\}$ if $x >_F y$, by $\{y\}$ if $x <_F y$, by $\{z\}$ for some $z$ in the $\sim_F$-equivalence class of $x$ and $y$ if $x \sim_F y$ but $x \neq -y$, and by some subset of that class together with $\{t \, | \, t <_F x\} \cup \{0\}$ if $x = -y$. This looks very similar to the hyperaddition given in Definition \ref{hypergroup.constr}.


We now want to consider the structure of the equivalence classes. Let $g$ be an equivalence class in $G$ and let $F_g$ be the set $g \cup \{0\}$. We can define a multivalued binary operation $\boxplus_g$ on $F_g$ by $x \boxplus_g y = (x \boxplus y) \cap F_g$.

\begin{lem}\label{subgroup.eqv.hypergroup}
For any element $g$ in $G$, $F_g$ is again a hypergroup, with hyperaddition given by $\boxplus_g$.
\end{lem}
\begin{proof} For every $x \in F_g$, we have $0\boxplus_g x = \{x\}\cap F_g = \{x\}$. 

Suppose $0\in x\boxplus_g y$, then $0\in x\boxplus y$, and so $y = -x$. Similarly, if $0\in y \boxplus_g x$, then $y = -x$.

For invertibility of sums, let $x, y, z \in F_g$ with $x\in y \boxplus_g z$. Then we have $x\in y \boxplus z$. By invertibility of sums of $F$, $-x \in -z \boxplus -y$. So $-x \in -z \boxplus_g -y$.

For associativity, suppose we have $x, y, z\in F_g$. We would like to show that
$$(x\boxplus_g y) \boxplus_g z = x \boxplus_g (y\boxplus_g z).$$

Let $t\in F_g$. Let us first show that $t\in x \boxplus_g (y\boxplus_g z)$ if and only if $t\in x \boxplus (y\boxplus z)$. It is clear that $x \boxplus_g (y\boxplus_g z) \subseteq x \boxplus (y\boxplus z)$. So it suffices to prove the other direction. We suppose that $t\in x \boxplus (y\boxplus z)$. Then there exists $k\in F$ such that $k\in y\boxplus z$ and $t\in x\boxplus k$. If $k\in F_g$, then we are done. If not, we have $y = -z$ and $k<_F y$. So we also have $k<_F x$, and so $t = x \in x\boxplus_g 0 \subseteq x\boxplus_g ( y \boxplus_g z)$. Similarly, we can also get $t\in (x \boxplus_g y) \boxplus_g z$ if and only if $t\in (x \boxplus y) \boxplus z$. By associativity of $F$, $(x\boxplus y) \boxplus z = x \boxplus (y\boxplus z)$. So $(x\boxplus_g y) \boxplus_g z = x \boxplus_g (y\boxplus_g z).$
\end{proof}

\begin{lem}\label{Fg}
For any element $g$ in $G$, $F_g$ is either isomorphic to $\mathbb{K}$ or isomorphic to $\mathbb{S}$ or is a group.
\end{lem}

\begin{proof}
For any $y$ and any $x$ with $x \in y \boxplus -y$, we have $y \in x \boxplus y$ and so $x <_F y$ unless $x \in \{-y, 0, y\}$. So for any $y \in F_g$ we have $y \boxplus_g -y \subseteq \{-y, 0, y\}$. Now suppose that there is some $y \in F_g$ with $y \boxplus_g -y \neq \{0\}$. Then $y$ is nonzero and $y, -y \in -y \boxplus_g y$. Suppose for a contradiction that there is some $z \in F_g \setminus \{-y, 0, y\}$, and let $t$ be the unique element of $-y \boxplus z$. Then by Lemma~\ref{pm}, $t \notin \{y, -y\}$, since $z \not <_F -y$. So $y \boxplus t = \{z\}$. Thus 
$$y \in y \boxplus_g 0 \subseteq (-y \boxplus_g y) \boxplus_g (t \boxplus_g -t) = -y\boxplus_g (y \boxplus_g t) \boxplus_g -t = -y \boxplus_g z \boxplus_g -t =t \boxplus_g -t,$$
and so $y \in \{-t, 0, t\}$, which is the desired contradiction.

So if there is any $y$ with $y \boxplus_g -y \neq \{0\}$, then $F_g = \{-y, 0, y\}$. It is now not hard to check that in this case if $y = -y$ then $F_g\cong \mathbb{K}$, and if $y \neq -y$ then $F_g \cong \mathbb{S}$. On the other hand, if there is no such $y$ then the hyperaddition on $F_g$ is single-valued, and so $F_g$ is a group.
\end{proof}

We can finally prove the Classification Theorem.

\begin{proof}[Proof of Theorem \ref{classgroups}]
Let $H$ be $F^{\times}$, let $G$ be given as above and let $\psi$ be the map sending an element $h$ of $H$ to its equivalence class in $G$. For any $x$ and $y$ in $H$, if $\psi(x) >_F' \psi(y)$ then $x >_F y$ and so $x \boxplus y = \{x\}$. Similarly if $\psi(x) <_F' \psi(x)$ then $x \boxplus y = \{y\}$. If $\psi(x) = \psi(y)$ then $x \boxplus_{\psi(x)} y =  (x \boxplus y) \cap F_{\psi(x)}$. So by the remarks following Lemma \ref{addition} we have that the hyperaddition of $F$ agrees with that of $\bigvee_{g\in G} F_g$ in this case as well.
\end{proof}

\section{Classification of stringent skew hyperfields}\label{sect.class.ssH}
In this section, we will present the classification of stringent skew hyperfields. We will first introduce a construction of skew hyperfields arising from short exact sequences.

\begin{defn} \label{def:hyp}
Let $F$ be a skew hyperfield and let $G$ be a totally ordered group. Suppose that we have a short exact sequence of groups $$1 \to F^{\times} \xrightarrow{\phi} H \xrightarrow{\psi} G \to 1\, .$$ Since $\phi$ is injective, by replacing $H$ with an isomorphic copy if necessary we may (and shall) suppose that $\phi$ is the identity. As usual, we define $x^h$ to be $h^{-1} \cdot x \cdot h$ for $x,h \in H$. We extend this operation by setting $0_F^h := 0_F$. We say that the short exact sequence {\bf has stable sums} if for each $h \in H$ the operation $x \mapsto x^h$ is an automorphism of $F$ (as a skew hyperfield). Since this operation clearly preserves the multiplicative structure, this is equivalent to the condition that it is always an automorphism of the underlying additive hypergroup. Furthermore, any short exact sequence as above with $H$ abelian automatically has stable sums. 

Suppose now that we have a short exact sequence with stable sums as above. Then we may define a hyperfield with multiplicative group $H$ as follows. We begin by choosing some object not in $H$ to serve as the additive identity, and we denote this object by $0$. For each $g$ in $G$, let $A_g$ be $\psi^{-1}(g) \cup \{0\}$. For any $h$ in $\psi^{-1}(g)$ there is a bijection $\lambda_h$ from $F$ to $A_g$ sending $0_F$ to $0$ and $x$ to $h \cdot x$ for $x \in F^{\times}$, and so there is a unique hypergroup structure on $A_g$ making $\lambda_h$ an isomorphism of hypergroups. Furthermore, this structure is independent of the choice of $h$ since for $h_1, h_2 \in \psi^{-1} (g)$ the map $\lambda_{h_1}^{-1}\cdot \lambda_{h_2}$ is just left multiplication by $h_1^{-1}\cdot h_2$, which is an automorphism of the additive hypergroup of $F$. In this way we obtain a well defined hypergroup structure on $A_g$, whose hyperaddition we denote by $\boxplus_g$.

Then the {\bf $G$-layering} $F \rtimes_{H, \psi} G$ of $F$ along this short exact sequence has as ground set $H \cup \{0\}$. Multiplication is given by $x \cdot y = 0$ if $x$ or $y$ is $0$ and by the multiplication of $H$ otherwise. $H\cup \{0\}$ is the underlying set of the hypergroup $\bigvee_{g\in G} A_g$, and we take the hyperaddition of $F \rtimes_{H, \psi} G$ to be given by that of this hypergroup. Explicitly; the hyperaddition is given by taking 0 to be the additive identity and setting
$$x \boxplus y = \begin{cases}
\{x\} & \text{if } \psi(x) > \psi(y), \\ 
\{y\} & \text{if } \psi(x) < \psi(y),\\ 
x \boxplus_{\psi(x)} y & \text{if } \psi(x) = \psi(y) \text{ and } 0 \not \in x \boxplus_{\psi(x)} y,\\ 
(x \boxplus_{\psi(x)} y) \cup \psi^{-1}( \psi(x)\downarrow) & \text{if } \psi(x) = \psi(y) \text{ and } 0 \in x \boxplus_{\psi(x)} y.
\end{cases}$$
\end{defn}

\begin{lem} \label{hyperfield.construct}
$F \rtimes_{H, \psi} G$ is again a skew hyperfield. If $F$ is stringent, then so is $F \rtimes_{H, \psi} G$.
\end{lem}
\begin{proof}
As shown in Lemma~\ref{comm.HG}, $\bigvee_{g\in G} A_g$ is a commutative hypergroup. So it suffices to show that $\cdot$ distributes over $\boxplus$. For left distributivity, we must prove an equation of the form $x_1 \cdot (x_2 \boxplus x_3) = x_1 \cdot x_2 \boxplus x_1 \cdot x_3$. As usual, if any of the $x_i$ is 0, then this is trivial, so we suppose that each $x_i$ is in $H$. If $\psi(x_2) > \psi(x_3)$, then both sides are equal to $x_1 \cdot x_2$. If $\psi(x_2) < \psi(x_3)$, then both sides are equal to $x_1 \cdot x_3$. So we may assume that $\psi(x_2) = \psi(x_3)$ and we call their common value $g$. Then $x_2 \boxplus_g x_3 = \lambda_{x_2}(1 \boxplus_F x_2^{-1} \cdot x_3)$ and $(x_1 \cdot x_2) \boxplus_{\psi(x_1) \cdot g} (x_1 \cdot x_3) = \lambda_{x_1 \cdot x_2}(1 \boxplus_F x_2^{-1}\cdot x_3)$. So if $0 \not \in x_2 \boxplus_g x_3$, then also $0 \not \in (x_1 \cdot x_2) \boxplus_{\psi(x_1) \cdot g} (x_1 \cdot x_3)$, and so both sides of the equation are equal to $x_1 \cdot (x_2 \boxplus_{g} x_3)$. If $0 \in x_2 \boxplus_g x_3$, then also $0 \in (x_1 \cdot x_2) \boxplus_{\psi(x_1)\cdot g} (x_1 \cdot x_3)$, and so both sides of the equation are equal to $x_1 \cdot (x_2 \boxplus_g x_3) \cup x_1\cdot \psi^{-1}(g\downarrow)$.

For the right distributivity, we need to consider bijections $\lambda_h' \colon F \to A_{\psi(h)}$ similar to the $\lambda_h$. We take $\lambda_h'(x)$ to be $x \cdot h$ for $x \in F^{\times}$ and to be $0$ for $x = 0_F$. Then since $\lambda_h'(x) = \lambda_h(x)^h$ for any $x$ and the short exact sequence has stable sums, the $\lambda_h'$ are also hyperfield isomorphisms. So we may argue as above but with the $\lambda_h'$ in place of the $\lambda_h$.

Finally, we must show that $F \rtimes_{H, \psi} G$ is stringent if $F$ is. By definition of $F \rtimes_{H, \psi} G$, we just need to show that for $x, y \in F \rtimes_{H, \psi} G$ with $\psi(x) = \psi(y)$ and $0 \not \in x \boxplus_{\psi(x)} y$, $x\boxplus y$ is a singleton. As $F$ is stringent and $0 \not \in x \boxplus_{\psi(x)} y$, then $x \boxplus_{\psi(x)} y$ is a singleton. So $x\boxplus y = x \boxplus_{\psi(x)} y$ is also a singleton.
\end{proof}

Now let us see some interesting examples of hyperfields constructed in this way.

\begin{example}
If $F$ is the Krasner hyperfield, $G$ and $H$ are both the additive group of real numbers, and $\psi$ is the identity, then $F \rtimes_{H, \psi} G$ is the tropical hyperfield.
\end{example}

\begin{example} \label{stringnotddEXS} $F: = \mathbb{Z}\cup \{- \infty\}$ in Example~\ref{stringentnotdd} can arise from the short exact sequence of groups
$$0 \to GF(2)^\times \xrightarrow{\phi} \mathbb{Z} \xrightarrow{\psi} \mathbb{Z} \to 0.$$
\end{example}

\begin{example} In \cite{AD18}, Anderson and Davis drew a diagram encoding many popular and important hyperfields and homomorphisms, see as follows. 

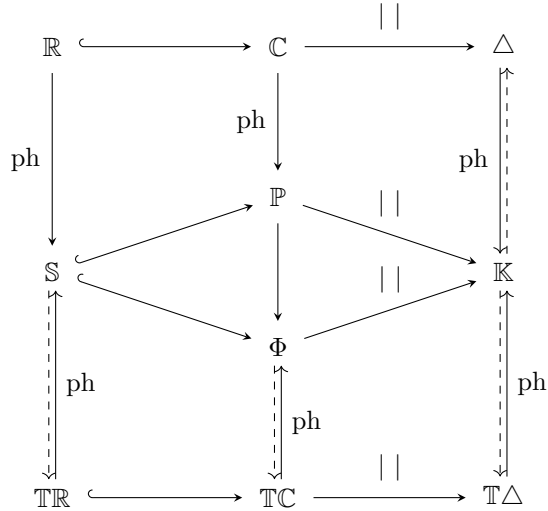
\begin{figure}[h]
\centering
\begin{tikzpicture}
\node (a) at (0, 0) {$\mathbb{TR}$}; \node (b) at (3, 0) {$\mathbb{TC}$}; \node (c) at (6, 0) {$\mathbb{T}\triangle$};
\node (d) at (0, 3) {$\mathbb{S}$}; \node (e) at (6, 3) {$\mathbb{K}$}; \node (f) at (3, 2) {$\Phi$}; \node (j) at (3, 4) {$\mathbb{P}$}; 
\node (g) at (0, 6) {$\mathbb{R}$}; \node (h) at (3, 6) {$\mathbb{C}$}; \node (i) at (6, 6) {$\triangle$};

\draw[right hook-stealth, shorten <= 3pt, shorten >= 3pt] (a) to (b);
\draw[-stealth, shorten <= 3pt, shorten >= 3pt] (b) to node [above=3pt] {$|\,\,|$} (c);
\draw[right hook-stealth, shorten <= 3pt, shorten >= 3pt] (g) to (h);
\draw[-stealth, shorten <= 3pt, shorten >= 3pt] (h) to node [above=3pt] {$|\,\,|$} (i);
\draw[right hook-stealth, shorten <= 3pt, shorten >= 3pt] (d) to (f);
\draw[-stealth, shorten <= 3pt, shorten >= 3pt] (f) to node [above=3pt] {$|\,\,|$} (e);
\draw[right hook-stealth, shorten <= 3pt, shorten >= 3pt] (d) to (j);
\draw[-stealth, shorten <= 3pt, shorten >= 3pt] (j) to node [above=3pt] {$|\,\,|$} (e);

\draw[-stealth, shorten <= 3pt, shorten >= 3pt] (j) to (f);

\draw[-stealth, shorten <= 3pt, shorten >= 3pt] (g) to node[left=1pt] {ph} (d);

\draw[-stealth, shorten <= 3pt, shorten >= 3pt] (h) to node[left=1pt] {ph} (j);

    \draw[transform canvas={xshift=0.3ex},->] (a) -- (d);
    \draw[transform canvas={xshift=-0.3ex},->, dashed] (d) to node [right=3pt] {ph} (a);
    
    \draw[transform canvas={xshift=0.3ex},->] (b) -- (f);
    \draw[transform canvas={xshift=-0.3ex},->, dashed] (f) to node [right=3pt] {ph} (b);
    
    \draw[transform canvas={xshift=0.3ex},->] (c) -- (e);
    \draw[transform canvas={xshift=-0.3ex},->, dashed] (e) to node [right=3pt] {ph} (c);
    
    \draw[transform canvas={xshift=0.3ex},->, dashed] (e) -- (i);
    \draw[transform canvas={xshift=-0.3ex},->] (i) to node [left=1pt] {ph} (e);

\end{tikzpicture}
\label{Diagram from AD}
\end{figure}

The diagram with the solid arrows commutes. The four dashed arrows are inclusions giving sections (one-sided inverses). Here ph is the \emph{phase map} ph$(x) =x/|x|$ if $x = 0$ and ph$(0) =0$. In each of the ten hyperfields, the underlying set is a subset of the complex numbers closed under multiplication. And in each hyperfield, multiplication, the additive identity, and the multiplicative identity coincides with that of the complex numbers.

Our classification gives a good relationship between the hyperfields in each column and we can construct each hyperfield in the bottom row from the corresponding element of the row just above the bottom and the ordered group $\mathbb{R}_{>0}$.
\begin{enumerate}
\item From the short exact sequence of groups
$$1 \to \mathbb{S}^\times \rightarrow \mathbb{R}^\times \rightarrow \mathbb{R}_{>0} \to 1,$$
we can get the tropical real hyperfield $\mathbb{TR} = \mathbb{S}\rtimes \mathbb{R}_{>0} $.

\item From the short exact sequence of groups
$$1 \to \Phi^\times \rightarrow \mathbb{C}^\times \rightarrow \mathbb{R}_{>0} \to 1,$$
we can get the tropical complex hyperfield $\mathbb{TC} = \Phi\rtimes \mathbb{R}_{>0}$.

\item From the short exact sequence of groups
$$1 \to \mathbb{K}^\times \rightarrow \mathbb{R}_{>0} \rightarrow \mathbb{R}_{>0} \to 1,$$
we can get the ultratriangle hyperfield $\mathbb{T}\triangle = \mathbb{K}\rtimes \mathbb{R}_{>0} $.
\end{enumerate}

Since in each column the second element is obtained as a quotient of the first by a $\mathbb{R}_{>0}-$subgroup, this operation of putting back the factor of $\mathbb{R}_{>0}$ yields a hyperfield on the same ground set as the top element of the column.
\end{example}

Our aim is to show that every stringent skew hyperfield is of the form $F \rtimes_{H, \psi} G$ with $F$ either the Krasner hyperfield or the sign hyperfield or a skew field. Let's start with a stringent skew hyperring.

Let $R$ be a stringent skew hyperring. By Theorem~\ref{classgroups}, we can classify $R$ to be the wedge sum $\bigvee_{g\in G} R_g$ with a surjective mapping $\psi$ from $R^\times$ to the set $G$ defined in last section and an ordering $<_R'$ on $G$ by $\psi(x) <_R' \psi(y)$ if and only if $x\boxplus y = \{y\}$ but $x\neq y$, where the hypergroup $R_g$ is either isomorphic to $\mathbb{K}$ or isomorphic to $\mathbb{S}$ or is a group. Thus by distributivity of $R$, $\psi(x) <_R' \psi(y)$ if and only if $\psi(ax) <_R' \psi(ay)$ if and only if $\psi(xa) <_R' \psi(ya)$ for $a\in R^\times$. So the multiplication of $R$ lifts to a multiplication on $G$ respecting the ordering, with identity $\psi(1): = 1_G$. By Lemma~\ref{eqv.set}(2), we can easily get the following lemma. 

\begin{lem}\label{eqv.monoid}
$(G, \cdot, <_R')$ is a totally ordered monoid. If $R$ is a skew hyperfield, then $G$ is a totally ordered group.
\end{lem}


Now we want to consider the structure of $R_g$. 

\begin{lem}\label{subgroup.eqv.hyperring}
$R_{1_G}$ is again a skew hyperring, with hyperaddition given by $\boxplus_{1_G}$ and multiplication by that of $R$.
\end{lem}

\begin{proof}
By Lemma~\ref{subgroup.eqv.hypergroup}, it suffices to check the distributivity. To prove left distributivity we must show that any element $t \in R_{1_G}$ of $x \cdot (y \boxplus z)$ is also an element of the same expression evaluated in $R_{1_G}$. So let $w$ be an element of $y \boxplus z$ with $x \cdot w = t$. This second equation implies that the equivalence class of $w$ is $1_G$, as desired. The right distributivity is similar.
\end{proof}

\begin{lem}\label{R1}
If $R$ is a skew hyperfield, $R_{1_G}$ is either the Krasner hyperfield or the sign hyperfield or a skew field.
\end{lem}

\begin{proof}
By Lemma~\ref{Fg} and Lemma~\ref{subgroup.eqv.hyperring}, we can get $R_{1_G}$ is either the Krasner hyperfield or the sign hyperfield or a skew ring.

Since $\sim_R$ respects the multiplication, the multiplicative inverse of anything equivalent to $1_R$ is again equivalent to $1_R$, so that $R_{1_G}$ is a skew field if it is a skew ring.
\end{proof}

\begin{lem} For every $g \in G$, the hypergroup of $R_g$ is isomorphic to the hypergroup of $R_{1_G}$.
\end{lem}
\begin{proof} 
Let $a\in R_g^\times$. Define $f: R_g \rightarrow R_{1_G}$ by sending $0$ to $0$ and $x$ in $R_g^\times$ to $a^{-1}\cdot x$. Since $f$ has an inverse operation, namely left multiplication by $a$, this is a bijection.
Now we would like to show $f(x\boxplus_g y) = f(x)\boxplus_{1_G} f(y)$. 
\begin{align*}
f(x\boxplus_g y) & = a^{-1}\cdot (x\boxplus_g y) \\
& = a^{-1}\cdot \big( (x\boxplus y)\cap R_g \big) \\
& = \big( a^{-1}\cdot (x\boxplus y) \big ) \cap (a^{-1}\cdot R_g) \\
& = \big( (a^{-1}\cdot x) \boxplus (a^{-1}\cdot y) \big) \cap R_{1_G} \\
& = (a^{-1}\cdot x) \boxplus_{1_G} (a^{-1}\cdot y) \\
& = f(x)\boxplus_{1_G} f(y) 
\end{align*}
\end{proof}

Now using the above results, we can classify the stringent skew hyperfields as follows.

\begin{thm}\label{class.ssH}
Any stringent skew hyperfield $R$ has the form $F \rtimes_{H, \psi} G$, where $F$ is either the Krasner hyperfield or the sign hyperfield or a skew field.
\end{thm}
\begin{proof}
Let $F$ be $R_{1_G}$, let $H$ be $R^{\times}$ and let $G$ be given as above. Let $\varphi$ be the injection of $F^{\times}$ as a subgroup of $H$ and let $\psi$ be the map sending an element $h$ of $H$ to its equivalence class in $G$. Then $$1 \to F^{\times} \xrightarrow{\phi} H \xrightarrow{\psi} G \to 1$$ is a short exact sequence. For any $x$ and $y$ in $H$, if $\psi(x) >_R' \psi(y)$, then $x >_R y$ and so $x \boxplus y = \{x\}$. Similarly if $\psi(x) <_R' \psi(x)$, then $x \boxplus y = \{y\}$. If $\psi(x) = \psi(y)$, then $x \boxplus_{\psi(x)} y = x \cdot ((1 \boxplus x^{-1} \cdot y) \cap R_{1_G}) = (x \boxplus y) \cap R_{\psi(x)}$. So by the remarks following Lemma \ref{addition} we have that the hyperaddition of $R$ agrees with that of $F \rtimes_{H, \psi} G$ in this case as well.
\end{proof}

Using results of Marshall's paper \cite{Mar06}, we can show that the structure is even more constrained if the multiplication of $R$ is commutative (so that $R$ is a stringent hyperfield) and $R_{1_G}$ is the Krasner or the sign hyperfield.

\begin{prop}\label{sH.sign.a^2} Let $R$ be a stringent skew hyperfield with $R_{1_G}= \mathbb{S}$ and let $a \in R^{\times}-\{1,-1\}$. Then $a^2 \notin \{1, -1\}$.
\end{prop}

\begin{proof} As $a\notin \{1, -1\}$, then $a\not\sim_R 1$. So $\psi(a) \neq 1$. Then $\psi(a^2) = (\psi(a))^2 \neq 1$, since $G$ is a totally ordered group. So $a^2\not\sim_R 1$. That is $a^2 \notin \{1, -1\}$.
\end{proof}

Following are some useful Lemmas in Marshall's paper (cf. section 3,\cite{Mar06}).

\begin{defn}\label{ordering}\cite{Mar06} Let $R$ be a hyperfield. A subset $P$ of $R$ is called an \textbf{ordering} if $$P \boxplus P \subseteq P, P \odot P \subseteq P, P\cup -P = R \text{ and }P \cap -P =\{0\}.$$
\end{defn}

\begin{defn}\label{real}\cite{Mar06} A hyperfield $R$ is said to be \textbf{real} if $-1 \notin R^2 \boxplus R^2$ where $R^2:=\{a^2 \, | \, a\in R\}$.
\end{defn}

\begin{lem}\label{ordering.real} \cite[Lemma 3.3]{Mar06} Let $R$ be a hyperfield. $R$ has an ordering if and only if $R$ is real.
\end{lem}

\begin{lem}\label{ordering.-1}\cite[Lemma 3.2, 3.3]{Mar06} Let $R$ be a hyperfield with $1\neq -1$. If $R$ has an ordering $P$, then $-1\notin P$.
\end{lem}

Based on above lemmas, we get the following.

\begin{prop} \label{sH.sign.ordering} If $R$ is a stringent hyperfield with $R_{1_G} = \mathbb{S}$, then $R$ has an ordering. 
\end{prop}

\begin{proof} By Lemma~\ref{ordering.real}, we just need to show that $R$ is real.

Suppose that $-1 \in R^2 \boxplus R^2$. Then there exist $a, b \in R$ such that $-1\in a^2 \boxplus b^2$. By Proposition~\ref{sH.sign.a^2}, $a^2 \neq -1$ and $b^2\neq -1$. Thus $a\neq 0$ and $b\neq 0$. And by reversibility, $-b^2 \in 1\boxplus a^2$. As $a^2 \neq -1$, then $1\boxplus a^2 \subseteq \{1, a^2\}$. Thus $-b^2 = a^2$. Then $-1 = a^2b^{-2} = (ab^{-1})^2$, a contradiction to Proposition~\ref{sH.sign.a^2}.
	
So $R$ is real, and therefore has an ordering. 
\end{proof}

\begin{thm}\label{K.S.constrain} If $R$ is a stringent hyperfield with $R_{1_G} \in \{ \mathbb{K}, \mathbb{S}\}$, then $R$ arises from a short exact sequence
$$1 \to R_{1_G}^\times \xrightarrow{\phi}  R_{1_G}^\times \times G \xrightarrow{\psi} G \to 1.$$
\end{thm}
\begin{proof} If $R_{1_G} = \mathbb{K}$, this is trivial.

 If $R_{1_G} = \mathbb{S}$, by Theorem~\ref{class.ssH}, we may suppose $R = \mathbb{S} \rtimes_{H, \psi} G = H\cup \{0\}$ with a short exact sequence of groups
 $$1 \to \mathbb{S}^\times \xrightarrow{\phi} H \xrightarrow{\psi} G \to 1.$$
 
By Proposition~\ref{sH.sign.ordering}, we know that $R$ has an ordering $P$. Let $O = P-\{0\}$. As $P\cup -P=R$ and $P \cap -P =\{0\}$, then $R = O\cupdot -O \cup \{0\}$. By Lemma~\ref{ordering.-1}, $-1\notin P$. Then $-1\notin O$, thus $1\in O$. And as $P\odot P \subseteq P$, then $O\odot O\subseteq O$. For any $a\in O$, $a^{-1} \in O$. Otherwise $-a^{-1} \in O$. Then $a\odot -a^{-1} = -1 \in O$, which is a contradiction. So $O$ is a multiplicative group with $1\in O$ and $R = O\cupdot -O \cup \{0\}$, and $\psi\restriction O$ is an isomorphism from $O$ to $G$.

Now we can identify $x\in H$ with $(1, \psi(x))$ if $x\in O$, and with $(-1, \psi(x))$ if $x\notin O$, giving a bijection from $H$ to $\mathbb{S}^\times \times G$.

So $R\cong (\mathbb{S}^\times \times G) \cup \{0\}$.

\end{proof}

It is not clear whether this result extends to stringent skew hyperfields.

\section{Classification of Doubly Distributive Skew Hyperfields}\label{sect.class.DDsH} 

In this section, we will present the classification of doubly distributive skew hyperfields.

\begin{prop}\label{DD.class} The doubly distributive skew hyperfields are precisely those of the form $F \rtimes_{H, \psi} G$ of exactly one of the following types:
\begin{enumerate}
\item $F$ is the Krasner hyperfield,

\item $F$ is the sign hyperfield,

\item $F$ is a skew field and $G$ satisfies 
$$\{ab \, | \, a, b <1_G \} = \{c \, | \,c<1_G \}.$$

\end{enumerate}
\end{prop}

The following example is a doubly distributive hyperfield of type (3) in Proposition~\ref{DD.class}.
\begin{example} Let $F: = \mathbb{R}$ be the hyperfield with usual multiplication and hyperaddition given by
$$x\boxplus y = \begin{cases}
x & \text{ if } |x|> |y|,\\
y & \text{ if } |x|< |y|,\\
\{z \,|\, |z|<|x|\} & \text{ if } |x| = |y|.
\end{cases}
$$
This hyperfield is stringent and arises from the short exact sequence of groups
$$1 \to GF(3)^\times \xrightarrow{\phi} \mathbb{R}^\times \xrightarrow{\psi} \mathbb{R}_{>0} \to 1\, .$$
\end{example}


Another natural example of doubly distributive skew hyperfields of type (3) in Proposition~\ref{DD.class} can be found in \cite{Pen18}, which is built around a (noncommutative) such hyperfield (the one called $L^\sigma$).

Before showing the proof of Proposition~\ref{DD.class}, We'll first introduce a useful lemma.

\begin{lem}\label{dd.1.-1} Let $R$ be a stringent skew hyperfield. $R$ is doubly distributive if and only if 
$$(1\boxplus -1) (1\boxplus -1) = 1\boxplus -1 \boxplus 1\boxplus -1.$$
\end{lem}

\begin{proof} By Definition~\ref{dd}, $R$ is doubly distributive if and only if $(a\boxplus b) (c\boxplus d) = ac \boxplus ad \boxplus bc \boxplus bd$, for any $a, b, c, d \in R$.

As $R$ is stringent, we have $u\boxplus v$ is a singleton if $u\neq -v$. So if either $a\neq -b$ or $c\neq -d$, then the equation above is just about distributivity. It already holds.

If both $a = -b$ and $c = -d$, then 
$$(a\boxplus b) (c\boxplus d) = (a\boxplus -a) (c\boxplus -c) = a (1\boxplus -1) (1\boxplus -1)c $$
and
$$ ac \boxplus ad \boxplus bc \boxplus bd = ac \boxplus -ac \boxplus -ac \boxplus ac = a (1\boxplus -1 \boxplus 1\boxplus -1) c.$$

So $R$ is doubly distributive if and only if 
$$(1\boxplus -1) (1\boxplus -1) = 1\boxplus -1 \boxplus 1\boxplus -1.$$
\end{proof}

Now we will present the proof of Proposition~\ref{DD.class}.

\begin{proof}[Proof of Proposition~\ref{DD.class}] 
By Proposition~\ref{ddsingleton} and Theorem~\ref{class.ssH}, we know that a doubly distributive skew hyperfield $R$ also has the form $F \rtimes_{H, \psi} G$, where $F$ is either the Krasner hyperfield or the sign hyperfield or a skew field, with a short exact sequence of groups
$$1 \to F^\times \xrightarrow{\phi} H \xrightarrow{\psi} G \to 1\, .$$ 
So it suffices to show that the hyperfields of type (1) and (2) are doubly distributive and the hyperfields with $F$ a skew field are doubly distributive if and only if they are of type (3).

\emph{Case 1:} When $F = \mathbb{K} = \{1, 0\}$, hyperaddition is defined by
\begin{align*}
x\boxplus 0 & = \{x\}, \\
x \boxplus y & =\begin{cases}
\{x\}	&\text{ if $ \psi(x) > \psi(y)$,}\\
\{y\}	&\text{ if $ \psi(x) < \psi(y)$,}\\
\{z\,|\, \psi(z)\leq \psi(x)\}\cup \{0\}	&\text{ if $ \psi(x) = \psi(y)$, that is $x=y$.}
\end{cases}
\end{align*}

By Lemma~\ref{dd.1.-1}, $R$ is doubly distributive if and only if
$$(1\boxplus 1) (1\boxplus 1) = 1\boxplus 1 \boxplus 1\boxplus 1.$$

$$(1\boxplus 1) (1\boxplus 1) = (\{z\,|\, \psi(z)\leq 1\}\cup \{0\}) \cdot (\{z\,|\, \psi(z)\leq 1\}\cup \{0\}) = \{z\,|\, \psi(z)\leq 1\}\cup \{0\},$$
and
$$1\boxplus 1 \boxplus 1\boxplus 1 = (\{z\,|\, \psi(z)\leq 1\}\cup \{0\}) \boxplus (\{z\,|\, \psi(z)\leq 1\}\cup \{0\}) = \{z\,|\, \psi(z)\leq 1\}\cup \{0\}.$$

So $R$ is doubly distributive when $F=\mathbb{K}$.

\emph{Case 2:} When $F = \mathbb{S}=\{1,-1,0\}$, hyperaddition is defined by
\begin{align*}
x\boxplus 0 & = \{x\}, \\
x \boxplus y & =\begin{cases}
\{x\}	&\text{ if $ \psi(x) > \psi(y)$,}\\
\{y\}	&\text{ if $ \psi(x) < \psi(y)$,}\\
\{x\}	&\text{ if $x=y$,}\\
\{z\,|\, \psi(z)\leq \psi(x)\}\cup \{0\}	&\text{ if $x=-y$.}
\end{cases}
\end{align*}

By Lemma~\ref{dd.1.-1}, $R$ is doubly distributive if and only if
$$(1\boxplus -1) (1\boxplus -1) = 1\boxplus -1 \boxplus 1\boxplus -1.$$

$$(1\boxplus -1) (1\boxplus -1) = (\{z\,|\, \psi(z)\leq 1\}\cup \{0\}) \cdot (\{z\,|\, \psi(z)\leq 1\}\cup \{0\}) = \{z\,|\, \psi(z)\leq 1\}\cup \{0\},$$
and
$$1\boxplus -1 \boxplus 1\boxplus -1 = (\{z\,|\, \psi(z)\leq 1\}\cup \{0\}) \boxplus (\{z\,|\, \psi(z)\leq 1\}\cup \{0\}) = \{z\,|\, \psi(z)\leq 1\}\cup \{0\}.$$

So $R$ is doubly distributive when $F=\mathbb{S}$.

\emph{Case 3:} When $F$ is a skew field, hyperaddition is defined by
\begin{align*}
x\boxplus 0 & = \{x\}, \\
x \boxplus y & =\begin{cases}
\{x\}	&\text{ if $ \psi(x) > \psi(y)$,}\\
\{y\}	&\text{ if $ \psi(x) < \psi(y)$,}\\
x\boxplus_{\psi(x)} y	&\text{ if $ \psi(x) = \psi(y)$ and $0\notin x\boxplus_{\psi(x)} y$,}\\
\{z\,|\, \psi(z)<\psi(x)\}\cup\{0\}	&\text{ if $ \psi(x) = \psi(y)$ and $0\in x\boxplus_{\psi(x)} y$.}\\
\end{cases}
\end{align*}

By Lemma~\ref{dd.1.-1}, $R$ is doubly distributive if and only if
$$(1\boxplus -1) (1\boxplus -1) = 1\boxplus -1 \boxplus 1\boxplus -1.$$

$$(1\boxplus -1) (1\boxplus -1) = (\{z\,|\, \psi(z)< 1\}\cup \{0\}) \cdot (\{z\,|\, \psi(z)< 1\}\cup \{0\}) = \{xy \,|\, \psi(x), \psi(y)<1 \}\cup \{0\},$$
and
$$1\boxplus -1 \boxplus 1\boxplus -1 = (\{z\,|\, \psi(z)< 1\}\cup \{0\}) \boxplus (\{z\,|\, \psi(z) < 1\}\cup \{0\}) = \{z\,|\, \psi(z)< 1\}\cup \{0\}.$$

So $R$ is doubly distributive if and only if 
$$\{xy \,|\, \psi(x), \psi(y)<1 \}\cup \{0\}  = \{z\,|\, \psi(z)< 1\}\cup \{0\}.$$

We claim that $$\{xy \,|\, \psi(x), \psi(y)<1 \}\cup \{0\} = \psi^{-1}(\psi(1)\downarrow) \cup \{0\} = \{z\,|\, \psi(z)< 1\}\cup \{0\},$$
if and only if
$$\{ab \, | \, a, b <1_G \} = \{c \, | \,c<1_G \}.$$

$(\Rightarrow): $ If $\{xy \,|\, \psi(x), \psi(y)<1 \}\cup \{0\} = \{z\,|\, \psi(z)< 1\}\cup \{0\}$, the direction $\subseteq$ is clear. We just need to consider the other direction. Let $c\in G$ be such that $c<1_G$ and let $z \in \psi^{-1}(c)$. Then there exist $x, y\in H$ such that $z = xy$ and $ \psi(x), \psi(y)<1$ by our assumption. So $c = \psi(z) = \psi(xy) = \psi(x)\psi(y)$. We have $c \in \{ab \, | \, a, b <1_G \}$.

$(\Leftarrow): $ If $\{ab \, | \, a, b <1_G \} = \{c \, | \,c<1_G \}$, the direction $\subseteq$ is also clear. We just need to consider the other direction. Let $z\in H$ be such that $\psi(z)< 1$ and let $c = \psi(z)$. Then there exist $a,b\in G$ such that $c = ab$ and $a,b<1_G$ by our assumption. Let $x\in H$ be such that $\psi(x) = a<1_G$ and let $y = x^{-1}z$. We have $\psi(y) = \psi(x^{-1}z) = a^{-1}c = b<1_G$ and $z = xy$. So $z\in \{xy \,|\, \psi(x), \psi(y)<1 \}$.

So $R$ is doubly distributive if and only if 
$$\{ab \, | \, a, b <1_G \} = \{c \, | \,c<1_G \}.$$
\end{proof}

\section{Reduction of stringent skew hyperrings to hyperfields}\label{sect.class.SSHR}

In this section, we will show that stringent skew hyperrings are very restricted.

\begin{thm}\label{hypring.class} Every stringent skew hyperring is either a skew ring or a stringent skew hyperfield.
\end{thm}
\begin{proof} If $G$ is trivial, then $R = R_{1_G}$. So $R$ is either $\mathbb{K}$, or $\mathbb{S}$, or a skew ring.

If $G$ is nontrivial, we would like to show that every element $x$ in $R^\times$ is a unit. Now let $s$ and $t$ in $R^\times$ be such that $x\cdot s >_R 1$ and $t \cdot x >_R 1$. Then by the remarks after Lemma~\ref{addition}, we have 
$$1 \in x\cdot s \boxplus -x\cdot s = x\cdot (s \boxplus -s),$$
$$1 \in t \cdot x \boxplus -t \cdot x = (t\boxplus -t) \cdot x.$$
So there exists $y \in s \boxplus-s$ and $z \in t \boxplus -t$ such that $1= x\cdot y = z \cdot x$. Thus $y = (z \cdot x) \cdot y = z\cdot (x \cdot y) = z$. So $x$ has a multiplicative inverse $y$ in $R$. Then $x$ is a unit of $R$.

So every stringent skew hyperring is either a skew ring or a stringent skew hyperfield.
\end{proof}

We cannot classify doubly distributive hyperring using our classification because not every doubly distributive hyperring is stringent. The following is a counterexample.

\begin{example}\label{coutex.ddhyperringnotstringent} The hyperring $\mathbb{K} \times \mathbb{K}$ that is the square of the Krasner hyperfield is doubly distributive but not stringent.
\end{example}

\section{Every stringent skew hyperfield is a quotient of a skew field}\label{sect.quotient} In this section, we would like to show that every stringent skew hyperfield is a quotient of a skew field by some normal subgroup. In particular, every stringent hyperfield is a quotient of a field by some special kind of subgroups, called `{\em h\"ullenbildend}'. This kind of subgroups was studied by Diller and Grenzd\"orffer in \cite{DG73} when they tried to unify the treatment of various notions of convexity in projective spaces over a field $K$ by introducing for any subgroup $U\leq K^{\times}$ the notion of $U$-convexity. They showed that this notion is reasonably well behaved if and only if $U$ is as follows.

\begin{defn}\cite{DG73} Let $K$ be a field and let $U\leq K^{\times}$. $U$ is called $U$-`{\bf h\"ullenbildend}' (hull producing) if $U$ satisfies
\begin{equation}\label{dress}
x, y \in K, x+y-xy \in U \rightarrow x \in U \text{ or } y \in U.
\end{equation}
\end{defn}

In \cite{Dre77}, Dress presented a simple complete classification of such `h\"ullenbildend' subgroups $U$. We will combine our classification of stringent (skew) hyperfields into three types with Dress's classification of such subgroups.

\begin{thm}\cite[Theorem 1]{Dre77}\label{hullen} Let $U\leq K^{\times}$ satisfy (\ref{dress}) and let $S_U=\{x\in K \, | \, x \notin U \text{ and } x + U \subseteq U \}$. Then $S_U$ is the maximal ideal of a valuation ring $R = R_U(=\{ x \in K  \, | \, x\cdot S_U \subseteq S_U \})$ in $K$, $U$ is contained in $R$, $\overline{U} = \{\overline{x} \in \overline{K}_U = R_U / S_U  \, | \, x \in U \}$ is either a domain of positivity in $\overline{K}_U$ (if $-1 \notin U$, $2 \in U $) or $\overline{U} = \{\overline{1}\}$ or $\overline{U} = \overline{K}_U^{\times}$ and, in any case, $U = \{x \in R_U  \, | \, \overline{x} \in \overline{U}\}$.
\end{thm}

We will first explain how to choose the suitable subgroup $U$ in the case of stringent hyperfield and then give the proof in full generality for stringent skew hyperfields.

From our classification in Theorem~\ref{class.ssH}, we know that every stringent hyperfield $F$ has the form $M \rtimes_{H, \psi} G$, where $M$ is either $\mathbb{K}$, or $\mathbb{S}$, or a field. To show that $F$ is a quotient by some subgroup $U$, we will choose $U$ with $\overline{U} = \{\overline{1}\}$ if $M$ is $\mathbb{K}$, choose $U$ with $\overline{U}$ a domain of positivity if $M$ is $\mathbb{S}$, and choose $U$ with $\overline{U} =\overline{K}_U^{\times}$ if $M$ is a field.

Now we begin the proof that every stringent skew hyperfield is a quotient of a skew field $K$ by some normal subgroup $U$. First, we recall the definition of quotient hyperfield. The quotient hyperfield $K/U = \{[g] = gU \, | \, g \in K \}$ was introduced by Krasner in \cite{Kra83} with multiplication given by $[g] \cdot [h] = [gh]$, for $[g], [h] \in K/U$. Hyperaddition is given by $[g]\boxplus [0] = [g]$ and $[g] \boxplus [h] = \{[f] \subseteq K/U  \, | \, f \in gU + hU \}$, for $[g], [h] \in (K/U)^{\times}$. As the subgroup $U$ we are choosing would be normal, so this quotient works in the skew case.

We may suppose that a stringent skew hyperfield $F = M \rtimes_{H, \psi} G$ arises from a short exact sequence of groups
$$1 \to M^\times \xrightarrow{\phi} H \xrightarrow{\psi} G \to 1\, ,$$
where $G$ is a totally ordered group equipped with a total order $\leq$ and $M$ is either $\mathbb{K}$, or $\mathbb{S}$, or a skew field. We define an order $\leq'$ on $G$ such that $x \leq' y$ if and only if $y\leq x$. So $\leq'$ is also a total order on $G$. 
Similarly as in the non-skew case, we will also choose $U$ with $\overline{U} = \{\overline{1}\}$ if $M$ is $\mathbb{K}$, choose $U$ with $\overline{U}$ a domain of positivity if $M$ is $\mathbb{S}$, and choose $U$ with $\overline{U} =\overline{K}_U^{\times}$ if $M$ is a skew field.

Our difficulty now is to choose a suitable skew field $K$ for the quotient hyperfield corresponding to each $U$. We will introduce two different constructions of skew fields, as follows.

\begin{example}\label{constr.valuation}\cite{FS01} (Construction 1)
Let $k$ be an arbitrary field. Define $K = k((G))$ to be the ring of formal power series whose powers come from $G$, that is, the elements of $K$ are functions from $G$ to $k$ such that the support of each function is a well-ordered subset of $(G, \leq')$. Addition is pointwise, and multiplication is the Cauchy product or convolution, that is the natural operation when viewing the functions as power series
$$\sum _{{a\in G}}p(a)x^{a}$$
It is well known (and easy to check) that $K$ is a skew field. 
\end{example}

We will construct a skew field $K = k((G))$ as in Example~\ref{constr.valuation} by choosing $k$ to be an arbitrary field when $M$ is $\mathbb{K}$ and choosing $k$ to be the field $\mathbb{R}$ of real numbers (or any other ordered field) when $M$ is $\mathbb{S}$. 

The second construction is for a stringent skew hyperfield $F = M \rtimes_{H, \psi} G$ when $M$ is a skew field.

\begin{example} \label{construction2} (Construction 2)
We define $K = M[[G]]$ to be the set of formal sums of elements of $H$ all from different layers such that the support is well-ordered, that is, an element of $K$ is a function $p$ from $G$ to $H$ such that for any $g$ in $G$, $p(g) \in \psi^{-1}(g) \cup \{0\} = A_g$ and the support of each function is a well-ordered subset of $(G, \leq')$. As $M$ is a skew field and each $\lambda_h$ with $h\in H$ is an isomorphism of hypergroups, then $(A_g, \boxplus_g, 0)$ is always an abelian group. We claim that $K$ is a skew field, viewing functions as power series
$$\sum_{{a\in G}} p(a) x^a \,,$$
with addition $+_K$ given by
$$\sum_{{a\in G}} p(a) x^a +_K \sum_{{a\in G}} q(a) x^a = \sum_{{a\in G}} (p(a)\boxplus_a q(a)) x^a,$$
and the additive identity is $\sum_{{a\in G}} 0 x^a$. Multiplication $\cdot_K$ is given by
$$\Big(\sum_{{a\in G}} p(a) x^a\Big) \cdot_K \big(\sum_{{a\in G}} q(a) x^a\Big) = \sum_{{s\in G}} \Big(\underset{g\cdot_G h = s}{\underset{h\in \supp(q),}{\underset{g\in \supp(p),}{\boxplus_{s}}}} p(g)\cdot_H q(h)\Big) x^s,$$
and the multiplicative identity is $1 x^{1_G}$. Since the proof that this really gives a skew field is a long calculation and is very similar to that for $k((G))$, we do not give it here but in Appendix~\ref{appendix}.
\end{example}

Now we divide the proof into three cases and show that $F\cong K/U$ in each case.  For simplicity, we denote $\min(\supp(p))$ by $m_p$ for $p \in K^{\times}$.

\begin{enumerate}[\emph{\text{Case}} 1:]
\item If $M$ is $\mathbb{K}$, then let $U = \{p\in K^\times \, | \, m_p = 1_G\}$. It's easy to check that $U$ is normal.
Then the quotient hyperfield $K/U = \{ [q] = qU  \, | \, q\in K\}$ has 
\begin{align*}
[q] &= \{p \in K^\times \, | \, m_p = m_q\}, \\
[0] &= \{0_K\}.
\end{align*}
So we can identify $[q]$ in $(K/U)^\times$ with $m_q$ in $G$ and identify $[0]$ in $K/U$ with $0$. So we have $K/U\cong (G\cup \{0\}, \boxplus, \cdot)$ with multiplication given by
\begin{align*}
0\cdot g &= 0,\\
g \cdot h & = g \cdot_G h,
\end{align*}
where $g, h \in G$.
And hyperaddition is given by
\begin{align*}
g\boxplus 0 & = \{g\}, \\
g \boxplus h & = \begin{cases}
\{g\} & \text{ if $g <' h$, that is $g>h$, } \\
\{h\} & \text{ if $g >' h$, that is $g<h$, } \\
\{f\in G  \, | \, f \geq' g \}\cup \{0\} = \{f\in G  \, | \, f \leq g \}\cup \{0\} & \text{ if $g = h$,} 
\end{cases}
\end{align*}
where $g , h \in G$.

Now it is clear to see that $K/U\cong (G\cup \{0\}, \boxplus, \cdot) \cong \mathbb{K} \rtimes_{H, \psi} G = F$.

\item If $M$ is $\mathbb{S}$, $k = \mathbb{R}$ (or any other ordered field) and $K= k((G))$, then let $U = \{p\in K^\times \, | \, m_p = 1_G \text{ and } p(1_G) > 0\}$. It's easy to check that $U$ is normal. Then the quotient hyperfield $K/U = \{ [q] = qU  \, | \, q\in K\}$ has 
\begin{align*}
[q] & = \{p \in K^\times \, | \, m_p = m_q \text{ and } p( m_p)  > 0\} \text{ if } q(m_q)> 0, \\
[q] & = \{p \in K^\times \, | \, m_p  = m_q  \text{ and } p( m_p)  < 0\} \text{ if } q(m_q)< 0, \\
[0] & = \{0_K\}.
\end{align*}
We can identify $[q]$ in $(K/U)^\times$ with $(1, m_q)$ if $q(m_q) >0$, identify $[q]$ in $(K/U)^\times$ with $(-1, m_q)$ if $q(m_q) < 0$, and identify $[0]$ with $0$. So we have $K/U\cong ((\mathbb{S}^\times \times G )\cup \{0\}, \boxplus, \cdot)$ with multiplication given by
\begin{align*}
(r, g)\cdot 0 &= 0, \\
(r_1, g_1) \cdot (r_2, g_2) &= (r_1 \cdot_\mathbb{S} r_2, g_1 \cdot_G g_2),
\end{align*}
where $r, r_1, r_2 \in \mathbb{S}^\times$ and $g , g_1, g_2 \in G$.
And hyperaddition is given by
\begin{align*}
(r, g)\boxplus 0 & = \{(r, g)\}, \\
(r_1, g_1) \boxplus (r_2, g_2) & = \begin{cases}
\{(r_1, g_1)\} & \text{if $g_1 < ' g_2$, that is $g_1 > g_2$, } \\
\{(r_2, g_2)\} & \text{if $g_1 > ' g_2$, that is $g_1 < g_2$, } \\
\{(r_1, g_1)\} & \text{if $g_1 = g_2$ and $r_1 = r_2$ } \\
\{ (r, f)  \, | \, f \geq' g \}\cup \{0\} = \{ (r, f)  \, | \, f \leq g \}\cup \{0\} & \text{if $g_1 = g_2$ and $r_1 =- r_2$,} 
\end{cases}
\end{align*}
where $r, r_1, r_2 \in \mathbb{S}^\times$ and $g , g_1, g_2 \in G$.

So by Theorem~\ref{K.S.constrain}, $K/U \cong ((\mathbb{S}^\times \times G )\cup \{0\}, \boxplus, \cdot) \cong \mathbb{S} \rtimes_{H, \psi} G = F$. 

\item If $M$ is a skew field and $K = M[[G]]$, then let $U = \{p \in K^{\times}  \, | \, m_p = 1_G \text{ and } p(1_G) = 1 \}$. It's easy to check that $U$ is normal. Then the quotient hyperfield $K/U = \{ [q] = qU  \, | \, q\in K\}$ has 
\begin{align*}
[q] &= \{p \in K^\times \, | \, m_p  = m_q  \text{ and } p(m_p)  = q(m_q)\},\\
[0] &= 0_K. 
\end{align*}
We can identify $[q]$ in $(K/U)^\times$ with $q(m_q)$ in $H$ (clearly $\psi(q(m_q)) = m_q$) and identify $[0]$ with $0_F$. So we have $K/U\cong F$ with multiplication given by
\begin{align*}
[q]\cdot 0 &= 0, \\
[q] \cdot  [h]&= \{ p\in K^\times \, | \, m_p  = m_q\cdot_G m_h  \text{ and } p(m_p)  = q(m_q)\cdot_H h(m_h)\} = [p]
\end{align*}
Hyperaddition is given by
$$ [q]\boxplus 0 = [q],$$
\begin{align*}
&  [q] \boxplus [h] \\
= &  \begin{cases}
\{[q]\} & \text{if $m_q<' m_h$, that is $m_q> m_h$, } \\
\{[h]\} & \text{if $m_q>' m_h$, that is $m_q< m_h$, } \\
\{[p] = \{p\in K^\times \,|\, m_p = m_q \text{ and } p(m_p) = q(m_q)\boxplus_{m_q} h(m_h)\} \}& \text{if $m_q= m_h$ and $0 \notin q(m_q)\boxplus_{m_q} h(m_h)$,} \\
\{[p]\in (K/U)^\times \,|\, m_p >' m_q\}\cup \{0\} = \{[p]\in (K/U)^\times \,|\, m_p < m_q\}\cup \{0\}  & \text{if $m_q= m_h$ and $0 \in q(m_q)\boxplus_{m_q} h(m_h).$}
\end{cases}
\end{align*}
where $[q], [h]\in (K/U)^\times$.

So $K/U\cong M \rtimes_{H, \psi} G = F.$
\end{enumerate}
 
\begin{thm} Every stringent skew hyperfield is a quotient of a skew field.
\end{thm}

\begin{cor} Every doubly distributive skew hyperfield is a quotient of a skew field.
\end{cor}

It follows from the construction that the same statements with all instances of the word `skew' removed also hold. 

\appendix
\section{The construction 2 in Example~\ref{construction2} gives a skew field}\label{appendix}
\begin{lem} Let $F = M \rtimes_{H, \psi} G$ be a stringent skew hyperfield arising from a short exact sequence of groups
$$1 \to M^\times \xrightarrow{\phi} H \xrightarrow{\psi} G \to 1\, ,$$
where $G$ is a totally ordered group and $M$ is a skew field. Define $K = k[[G]]$ as we did in section~\ref{sect.quotient}. Then $K$ is a skew field.
\end{lem}

\begin{proof}
The commutativity and associativity of $(K, +_K, \sum_{{a\in G}} 0 x^a) $ follow from those of $(H\cup \{0\}, \boxplus, 0)$. So we only need to show the associativity of $(K, \cdot_K, 1 x^{1_G}) $, the existence of a multiplicative inverse for every element and the distributivity.

An important principle which we will need again and again as we go along is a kind of distributivity of the composition of $H$ over the various additions $\boxplus_g$. To express it cleanly, we begin by extending $\cdot_H$ to $H \cup \{0\}$ by setting $x \cdot 0 = 0 \cdot x = 0$ for all $x \in H \cup \{0\}$. Suppose that we have elements $x$ and $y_1, y_2 \ldots y_n$ of $H$ with $\psi(y_i) = u\in G$ for all $i$, so that $\boxplus_{i = 1}^n y_i$ is defined. Let $v\in G$ be such that $v = \psi(x)$. Then $z \mapsto x \cdot_H z$ is a bijection from $A_u$ to $A_{v \cdot u}$ whose composition with $\lambda_{y_1}$ is $\lambda_{x \cdot_H y_1}$, so it must also be an isomorphism of hypergroups. Thus
$$x \cdot_H \big(\underset{1\leq i \leq n}{\boxplus_u} y_i\big) =  \underset{1\leq i \leq n}{\boxplus_{v \cdot u}} x \cdot_H y_i \,.$$
A similar argument using the $\lambda'_h$ defined in the proof of Lemma~\ref{hyperfield.construct} shows
$$\big(\underset{1\leq i \leq n}{\boxplus_u} y_i\big) \cdot_H x = \underset{1\leq i \leq n}{\boxplus_{u \cdot v}} y_i \cdot_H x\,.$$

To show the associativity of $(K, \cdot_K, 1 x^{1_G})$, we let $p, q, w\in K$. Then for $s\in G$,
\begin{align*}
\big( (p\cdot_K q) \cdot_K w \big) (s) & =  \underset{g\cdot_G c = s}{\underset{c\in \supp(w)}{\boxplus_s}} \Big( \big( \underset{a \cdot_G b = g}{\underset{b\in \supp(q)}{\underset{a\in \supp(p)}{\boxplus_g}}} p(a)\cdot_H q(b) \big) \cdot_H w(c) \Big) \\
& = \underset{g\cdot_G c = s}{\underset{c\in \supp(w)}{\boxplus_s}} \Big( \underset{a \cdot_G b = g}{\underset{b\in \supp(q)}{\underset{a\in \supp(p)}{\boxplus_s}}} p(a)\cdot_H q(b)\cdot_H w(c) \Big)\\
& = \underset{a \cdot_G b \cdot_G c = s}{\underset{c\in \supp(w)}{{\underset{b\in \supp(q)}{\underset{a\in \supp(p)}{\boxplus_s}}}}} p(a)\cdot_H q(b)\cdot_H w(c), \\
& =  \underset{a \cdot_G h = s}{\underset{a\in \supp(p)}{\boxplus_s}} \Big( \underset{b \cdot_G c = h}{\underset{c\in \supp(w)}{\underset{b\in \supp(q)}{\boxplus_s}}} p(a) \cdot_H q(b)\cdot_H w(c) \Big) \\
& =  \underset{a \cdot_G h = s}{\underset{a\in \supp(p)}{\boxplus_s}} \Big( p(a) \cdot_H \big( \underset{b \cdot_G c = h}{\underset{c\in \supp(w)}{\underset{b\in \supp(q)}{\boxplus_h}}} q(b)\cdot_H w(c) \big)  \Big) \\
&=\big( p\cdot_K (q \cdot_K w)\big) (s) \,.
\end{align*}
So $(p\cdot_K q) \cdot_K w = p\cdot_K (q \cdot_K w)$.

Next we will show that each element of $K$ has a multiplicative inverse. We do this first for those $p\in K = k[[G]]$ such that $m_p = 1_G$ and $p(m_p) = 1$. Let $S$ be the set of finite sums of elements of $\supp(p)$. $S$ is well founded.

Define $q\in K= k[[G]]$ such that $q(1_G) := 1$, $q(s) := 0$ for $s\notin S$ and, for $s \in S$, define $q(s)$ recursively by
$$q(s) := - \Big(\underset{g\cdot_G h = s}{\underset{h\in S-\{s\}}{\underset{g\in \supp(p)-\{1_G\}}{\boxplus_s }}} p(g)\cdot_H q(h) \Big).$$
So
\begin{align*}
p\cdot_K q(1_G) & = 1, & \\
p\cdot_K q(s) & = 0 & \text{ if $s\notin S$,} \\
p\cdot_K q(s) & = \underset{g\cdot_G h = s}{\underset{h\in \supp(q)}{\underset{g\in \supp(p)}{\boxplus_s }}} p(g)\cdot_H q(h)\\
& = \Big(\underset{g\cdot_G h = s}{\underset{h\in \supp(q)-\{s\}}{\underset{g\in \supp(p)-\{1_G\}}{\boxplus_s }}} p(g)\cdot_H q(h)\Big) \boxplus_s p(1)\cdot_H q(s)& \\
& = \Big(\underset{g\cdot_G h = s}{\underset{h\in \supp(q)-\{s\}}{\underset{g\in \supp(p)-\{1_G\}}{\boxplus_s }}} p(g)\cdot_H q(h)\Big) \boxplus_s q(s) & \\
& = 0 & \text{ if $s\in S-\{1_G\}$.}
\end{align*}

So $p \cdot_K q$ is the identity. Therefore, $q$ is the multiplicative inverse of $p$.

Next we consider elements of $K$ with only a single summand, that is, those of the form $a x^g$. It is clear that each such element also has a multiplicative inverse, namely $a^{-1} x^{g^{-1}}$.

Now every element of $K$ can be expressed as a product $p_1 \cdot p_2$, with $m_{p_1} = 1_G$ and $p_1(m_{p_1}) = 1$ and such that $p_2$ has only a single summand. As seen above, each of $p_1$ and $p_2$ has a multiplicative inverse, and hence $p_1 \cdot p_2$ also has one, namely $p_2^{-1} \cdot p_1^{-1}$. 

For distributivity, we first would like to show that $p\cdot_K (q +_K w) = p\cdot_K q +_K p\cdot_K w$. For $s\in G$,
\begin{align*}
(p\cdot_K (q +_K w)) (s) & = \underset{g\cdot_G h = s}{\boxplus_s} p(g)\cdot_H (q(h)\boxplus_h w(h)) \\
& = \underset{g\cdot_G h = s}{\boxplus_s} \big(p(g)\cdot_H q(h) \boxplus_s p(g)\cdot_H w(h)\big)\\
& = \big(\underset{g\cdot_G h = s}{\boxplus_s} p(g)\cdot_H q(h)\big)\boxplus_s \big(\underset{g\cdot_G h = s}{\boxplus_s} p(g)\cdot_H w(h)\big),\\
&= (p\cdot_K q +_K p\cdot_K w)(s) \,.
\end{align*}
So $p\cdot_K (q +_K w) = p\cdot_K q +_K p\cdot_K w$. A similar calculation shows that $(p +_K q) \cdot_K w = p \cdot_K w +_K q \cdot_K w$. 

So $K=k[[G]]$ is a skew field.
\end{proof}

\section{The semirings associated to doubly distributive hyperfields} \label{app2}

In \cite{GJL17}, Lemma 6.2(2) provides a way to build a semiring out of a doubly distributive hyperfield.\footnote{In \cite[Theorem 2.5]{Row16}, Rowen extended the theory of constructing the semiring to every hyperfield.} In this section, we would like to talk about these semirings.

For any doubly distributive hyperfield $H$ we can define binary operations $\oplus$ and $\odot$ on $\mathcal{P} H$ by setting $A \oplus B := \bigcup_{a \in A, b \in B} a \boxplus b$ (this is just the extension of $\boxplus$ to subsets of $H$ from Definition \ref{hyperoperation}) and $A \odot B := \{a b \colon a \in A, b \in B\}$. Let $\langle H \rangle$ be the substructure of $(\mathcal P H, \oplus, \odot)$ generated from the singletons of elements of $H$. 
So $\langle H \rangle$ is a semiring.
We will refer $\langle H \rangle$ as the {\em associated semiring} to $H$. Using our classification, we can easily determine all such associated semirings. Surprisingly, some of the basic examples have already been intensively studied and play an important role in the foundations of tropical geometry. In each case, we find that $\langle H \rangle$ contains only few elements in addition to the singletons of elements of $H$.

We have seen that any doubly distributive hyperfield has the form $F \rtimes_{H, \psi} G$, where $F$ is the Krasner hyperfield, the sign hyperfield or a field. We divide into cases according the value of $F$.

\subsection{Supertropical semirings} 
If $F$ is the Krasner hyperfield then $\psi \colon H^{\times} \to G$ is an isomorphism, and we can take it to be the identity. Then the elements of $\langle H \rangle$ are the singletons of elements of $H$ and the sets $g^{\nu} := \{h \in G \colon h \leq g\} \cup \{0\}$. To simplify the definition of the addition we define an operation $\nu$ on $\langle H \rangle \setminus \{\{0\}\}$ by $\nu(\{g\}) = \nu(g^{\nu}) = g^{\nu}$ and we transfer the total order of $G$ to the $g^\nu$ in the obvious way. Then addition is given by $x \oplus \{0\} = x$ for any $x$ and otherwise by 

$$x \oplus y =  \begin{cases} x & \text{if $\nu(x) > \nu(y)$,} \\
y & \text{if $\nu(x) < \nu(y)$,} \\
\nu(x) & \text{if $\nu(x) = \nu(y)$.}\\
\end{cases} $$

Multiplication is given by $x \odot \{0\} = \{0\}$, by $\{g\} \odot \{h\} = \{g \cdot h\}$, by $\{g\} \odot h^\nu = (g \cdot h)^{\nu}$ and by $g^\nu \odot h^\nu = (g \cdot h)^\nu$. In the case that $G$ is the ordered group of real numbers, this is simply the supertropical semiring introduced by Izhakian in \cite{Izh}. This associated semiring has also been studied by Rowen in \cite{Row16}.
It would be reasonable to call such semirings in general {\em supertropical semirings}.

\subsection{Symmetrised $(\max, +)$-semirings}
If $F$ is the sign hyperfield then by Theorem \ref{K.S.constrain} without loss of generality it arises from a short exact sequence $$ 1 \to \mathbb{S}^{\times} \to \mathbb{S}^{\times} \times G \to G \to 1\,.$$

The elements of $\langle H \rangle$ then have the form $0 := \{0_H\}$, $\oplus g := \{(1, g)\}$, $\ominus g := \{(-1, g)\}$, or $g^{\circ} := \{(i, h) \colon i \in \mathbb{S}^{\times}, h \leq g\}\cup \{0_H\}$. There is an obvious projection map $\pi$ from $\langle H \rangle \setminus \{0\}$ to $G$. Then addition is given by  $x \oplus 0 = x$ for any $x$, by $x \oplus y = x$ if $\pi(x) > \pi(y)$, by $x \oplus g^{\circ} = g^{\circ}$ if $\pi(x) = g$, by $(\oplus g) \oplus (\oplus g) = \oplus g$, by $(\ominus g) \oplus (\ominus g) = \ominus g$ and by $(\oplus g) \oplus (\ominus g) = g^{\circ}$. Multiplication is given by $x \odot 0 = 0$ for any $x$, by $x \odot g^{\circ} = (\pi(x) \cdot g)^{\circ}$, by $(\oplus g) \odot (\oplus h) = \oplus(g \cdot h)$, by $(\ominus g) \odot (\ominus h) = \oplus(g \cdot h)$ and by $(\oplus g) \odot (\ominus h) = \ominus (g \cdot h)$. 

In the case that $G$ is the ordered group of real numbers, this is simply the symmetrised $(\max, +)$-semiring introduced by Akian et al in \cite{plus}. So it would be reasonable to call such semirings in general {\em symmetrised $(max, +)$-semirings}.

\subsection{Linearised $(\max, +)$-semirings}
If $F$ is a field, then the elements of $\langle H \rangle$ are the singletons of elements of $H$ (which are in canonical bijection with $H$) and the sets $\psi^{-1}(g \downarrow) \cup \{0\}$ (which are in canonical bijection with $G$). So $\langle H \rangle$ is isomorphic to the semiring on $H \cup G$ with $x \oplus y$ for $x, y \in H$ given by the unique element of $x \boxplus y$ if this set is a singleton and by $\psi(x)$ otherwise, with $x \oplus g$ for $x \in H$ and $g \in G$ given by $x$ if $\psi(x) \geq g$ and by $g$ otherwise, and with $g \oplus h$ for $g, h \in G$ given by $\max(g, h)$. For multiplication, $x \odot y = x \cdot y$ for $x, y \in H$ and $x \odot g = \psi(x) \cdot g$ for $x \in H$ and $y \in G$ and finally $g \odot h = g \cdot h$ for $g, h \in G$.

By analogy to the previous construction, we could refer to such semirings as {\em linearised $(\max, +)$-semirings}. So far as we know, such semirings have not yet been seriously investigated.

\bibliographystyle{alpha}
\bibliography{biblioCM}

\end{document}